\documentclass[a4paper,11pt]{amsart}
\usepackage{a4wide}
\usepackage[utf8]{inputenc}
\usepackage{amssymb}
\usepackage{url}
\usepackage{hyperref}

\usepackage{tikz}
\usepackage{color}

\newcommand{\field}[1]{\mathbb{{#1}}}

\newcommand{\eps}{\varepsilon}

\newcommand{\N}{\field{{N}}}
\newcommand{\Z}{\field{{Z}}}
\newcommand{\Q}{\field{{Q}}}

\newcommand{\intpartup}[1]{\left\lceil#1\right\rceil}
\newcommand{\fracpart}[1]{\left\{#1\right\}}
\newcommand{\dd}{\,\mathrm{d}}

\DeclareMathOperator{\odd}{odd}
\DeclareMathOperator{\even}{even}
\DeclareMathOperator{\Ree}{Re}

\newtheorem*{theorem*}{Theorem}
\newtheorem{theorem}{Theorem}
\newtheorem{corollary}{Corollary}

\newtheorem{lemma}{Lemma}
\newtheorem*{lemma*}{Lemma}

\theoremstyle{remark}
\newtheorem*{acknowledgments}{Acknowledgments}

\allowdisplaybreaks[4]

\newenvironment{List}{\begin{list}{$\bullet$}{
\setlength{\labelwidth}{.5cm}
\setlength{\leftmargin}{.7cm}
}
}{\end{list}}     

\newcounter{fixedfig}
\newenvironment{fixedfig}{
\refstepcounter{fixedfig}
\centerline\bgroup%
\def\caption##1{\begin{minipage}{10.5cm}\textsc{Figure \arabic{fixedfig}}: ##1\end{minipage}}%
\begin{tabular}{c}%
}
{%
\end{tabular}%
\egroup%
}

\makeatletter
\renewcommand*\pmod[1]{%
  \allowbreak
  \mathchoice
    {\mkern  9mu}
    {\mkern  8mu}%
    {\mkern  6mu}
    {\mkern  6mu}%
  ({\operator@font mod}\mkern4mu #1)%
}
\makeatother

\author[F.~Battistoni]{Francesco Battistoni}
\address[F.~Battistoni]{Dipartimento di Matematica per le Scienze Economiche, Finanziarie ed Attuariali\\
         Universit\`{a} Cattolica\\
         via Necchi 9\\
         20123 Milano\\
         Italy}
\email{francesco.battistoni@unicatt.it}

\author[L.~Greni\'{e}]{Lo\"{\i}c Greni\'{e}}
\address[L.~Greni\'{e}]{Dipartimento di Ingegneria gestionale, dell'informazione e della produzione\\
         Universit\`{a} di Bergamo\\
         viale Marconi 5\\
         24044 Dal\-mi\-ne\\
         Italy}
\email{loic.grenie@gmail.com}

\author[G.~Molteni]{Giuseppe Molteni}
\address[G.~Molteni]{Dipartimento di Matematica\\
         Universit\`{a} di Milano\\
         via Saldini 50\\
         20133 Milano\\
         Italy}
\email{giuseppe.molteni1@unimi.it}

\keywords{Periods length of continued fractions, quadratic surds}

\subjclass[2020]{11A55, 11R11, 11Y65}

\title[First and second moment for the length]
      {The first and second moment for the length of the period of the continued fraction expansion for $\sqrt{d}$}

\begin{document}

\begin{abstract}
Let $d$ be any positive and non square integer. We prove an upper bound for the first two moments of the
length $T(d)$ of the period of the continued fraction expansion for $\sqrt{d}$. This allows to improve
the existing results for the large deviations of $T(d)$.
\end{abstract}

\maketitle

\begin{center}
Mathematika {\bf 70} (4), pp. 12 (2024).\\
Electronically published on July 23, 2024.\\
DOI: \url{https://doi.org/10.1112/mtk.12273}
\end{center}

\section{Introduction and results}
\label{sec:1} Let $d\in\N$. Let $T(d)$ be the length of the minimal positive period for the simple
continued fraction expansion for $\sqrt{d}$ when $d$ is not a square, otherwise let $T(d)=0$.
Sierpi\'{n}ski~\cite[p.~293]{Sierpinski3} (see also~\cite[p.~315]{Sierpinski}) noticed that Lagrange's
argument proving the periodicity of this representation actually proves that $T(d) \leq 2d$.
%
%
%
%
%
%
Hickerson~\cite{Hickerson} noticed that this argument shows that $T(d)\leq g(d)$, where
\[
g(d) := \#\big\{(m,q)\in\N^2\colon m<\sqrt{d},\ |q-\sqrt{d}|<m,q\mid (d-m^2)\big\},
\]
and used this bound to prove that $T(d)\leq g(d) \leq d^{1/2 + \log2/\log\log d + O(\log\log\log
d/(\log\log d)^2))}$. He also proved that $g(d) = \Omega(\sqrt{d})$. In the same years both
Hirst~\cite{Hirst} (for squarefree $d$, only) and Podsypanin~\cite{Podsypanin2,Podsypanin} (for any
$d$) independently proved that $T(d) \ll \sqrt{d}\log d$; Hirst as a consequence of a different formula
computing $g(d)$
%
%
%
%
allowing a stronger bound for $g(d)$, and Podsypanin via a connection between $T(d)$ and the residue at
$1$ of the Dedekind zeta function associated with the field $\Q[\sqrt{d}]$. The same connection was
independently noticed also by Stanton, Sudler and Willams~\cite{StantonSudlerWilliams} who proved that
$T(d) \leq c \sqrt{d}\log d$ for a smaller value of the constant $c$, later further improved by
Cohn~\cite{Cohn4}. This connection also shows that in case the Generalized Riemann Hypothesis holds for
quadratic fields, then the upper bound improves to $T(d)\ll \sqrt{d}\log\log d$
(see~\cite{Podsypanin2,Podsypanin}), in agreement with the heuristic and numerical investigations
in~\cite{Williams2}. Cohn~\cite{Cohn4} also proved that $T(d) = \Omega (\sqrt{d}/\log\log d)$. \\
At first it was believed that $T(d)\ll \sqrt{d}$, but nowadays the general sentiment has changed, in
spite of the fact that all proven results are still compatible with such a strong upper bound. For sure
$T(d)$ strongly oscillates, passing from values as small as $1$ (for $d=m^2+1$, any $m$) and $2$ (for
$d=m^2+2m/a$, any $m$ and any $a|2m$, $a\neq 2m$) to as large as $\sqrt{d}/\log\log d$ infinitely often
by Cohn's theorem. This suggests to try to bound in some way the number of $d$ where $T(d)$ is
exceptionally large, in some sense. Let
\[
D(x,\alpha) := \{d\in (x,2x]\colon T(d)>\alpha \sqrt{d}\}.
\]
Rockett and Sz\"{u}sz~\cite{RockettSzusz2} used the connection with the residue of the Dedekind zeta
function to prove that
\begin{equation}\label{eq:1A}
\# D(x,\alpha) \leq \frac{c+o(1)}{\log^2\alpha}\, x
\qquad
\text{as $x\to\infty$}
\end{equation}
for a suitable but undetermined constant $c$.

In this paper we extract some more information from the bound $T(d)\leq g(d)$, via an explicit
computation of the first two moments for $g$. In fact, we prove the following facts.
\begin{theorem}\label{th:1}
Let $x>1$, then
\begin{align}
&\sum_{d\leq x} g(d) = c_1x^{3/2} - 2x - 2\sqrt{x}  + \theta (x+4\sqrt{x}),   \label{eq:2A}
\intertext{where $c_1 := \frac{4}{3}\log 2 = 0.9241\ldots$ and $\theta = \theta(x)\in [0,1]$, and}
&\sum_{d\leq x} g(d)^2 \leq 11.9\,x^2 + 5\,x^{3/2}\log^2(4e^4x).              \label{eq:3A}
\end{align}
\end{theorem}
\noindent %
A modification of the argument proving~\eqref{eq:3A} allows to strengthen its conclusion and prove that
$\sum_{d\leq x} g(d)^2$ $\leq 8 x^2 + O(x^{3/2}\log^4 x)$ as $x$ diverges: we do not provide the details
for this improvement, but we describe the basic steps at the end of Section~\ref{sec:3}.\\
From the inequality $T(d)\leq g(d)$ we deduce that:
\begin{corollary}\label{cor:1A}
Let $x>1$, then
\begin{align*}
\sum_{d\leq x} T(d) \leq c_1 x^{3/2}
\quad\text{and}\quad
\sum_{d\leq x} T(d)^2 \leq 11.9\,x^2 + 5\,x^{3/2}\log^2(4e^4x)
\end{align*}
where $c_1 := \frac{4}{3}\log 2 = 0.9241\ldots$.
%
Moreover,
\begin{align*}
\sum_{x< d\leq 2x} T(d) \leq c_2 x^{3/2} + O(\sqrt{x})
\quad\text{and}\quad
\sum_{x< d\leq 2x} T(d)^2 \leq 47\,x^2 + O(x^{3/2}\log^2x)
\end{align*}
as $x\to\infty$, where $c_2 := \frac{8\sqrt{2}-4}{3}\log 2 = 1.6898\ldots$.
\end{corollary}
\begin{proof}
Only the last claim is not immediate. It follows noticing that $\sum_{x< d\leq 2x} g(d)^2 = \sum_{d\leq
2x} g(d)^2 - \sum_{d\leq x} g(d)^2$, that $\sum_{x< d\leq 2x} g(d)^2 \leq 47.6\,x^2 + O(x^{3/2}\log^2x)$
by~\eqref{eq:3A}, and that $\sum_{d\leq x} g(d)^2\geq c_1^2\,x^2 + O(x^{3/2})$ by the Cauchy--Schwarz
inequality and~\eqref{eq:2A}.
%
\end{proof}
\noindent %
The bounds in Corollary~\ref{cor:1A} allow to improve~\eqref{eq:1A} in the following way.
\begin{corollary}\label{cor:2A}
Let $\alpha >0$ and $x>1$. Then
\[
\# D(x,\alpha) \leq \frac{c_2 + o(1)}{\alpha}\, x
\quad\text{and}\quad
\# D(x,\alpha) \leq \frac{47 + o(1)}{\alpha^2}\, x,
\]
where $c_2 := \frac{8\sqrt{2}-4}{3}\log 2 = 1.6898\ldots$.
\end{corollary}
The claims are not trivial for $\alpha > c_2$ and $\alpha > \sqrt{47}$, respectively; the first claim
supersedes the second one when $\alpha \in (\sqrt{47},47/c_2)$.
\medskip
\medskip\\
We have some comments about these results:
\begin{List}
\item %
    For the second moment we have not determined the asymptotic behaviour, but the order we have found is
    the correct one since, as we have already noticed, the Cauchy--Schwarz inequality and~\eqref{eq:2A}
    show that
    \[
    \sum_{d\leq x} g(d)^2 \geq (c_1^2+o(1))x^2
    \quad\text{as $x\to\infty$.}
    \]
\item %
    In similar way, H\"{o}lder's inequality shows that every moment $M_{g,r}$ satisfies the lower-bound
    \[
    M_{g,r}(x)
    :=\sum_{d\leq x} g(d)^r \geq (c_1^r+o_r(1))x^{r/2+1}
    \quad\text{as $x\to\infty$.}
    \]
    For some time we have cultivated the hope to prove that this is the correct order also for $r\geq 3$.
    In fact, we have succeeded to prove that the analogue of the $U$ term appearing in the proof
    of~\eqref{eq:3A} (see Section~\ref{sec:3}) is $\ll_r x^{r/2+1}$ for every $r$. However, the analogue
    of the $V$ term resisted to every attempt to bound it nontrivially, and numerically, for $r=3$, it
    seems that it grows as $x^{5/2}\log x$, not as $x^{5/2}$. As a consequence, the idea that the order
    of $M_{g,r}(x)$ is $x^{r/2+1}$ is probably incorrect, and a behaviour of type $M_{g,r}(x)\asymp_r
    x^{r/2+1}(\log x)^{r-2}$ appears to be a better conjecture.
\item %
    An upper bound of that type can be proved for the moments of $T$, since
    \begin{align*}
    M_{T,r}(x)
    &:=\sum_{d\leq x} T(d)^r
      =\sum_{d\leq x} T(d)^{r-2} T(d)^2\\
    &\ll_r x^{(r-2)/2}(\log x)^{r-2}\sum_{d\leq x} T(d)^2
     \ll_r x^{r/2+1}(\log x)^{r-2}
    \end{align*}
    for every $r\geq 2$, by the bound for $T(d)$ and Corollary~\ref{cor:1A}. (This argument cannot by
    applied for the moments of $g$ since the bound $\sqrt{d}\log d$ is proved for $g(d)$ only when $d$ is
    squarefree).
    However, the moments for $T$ are probably smaller. In fact, computations support the conjecture that
    the orders of both the mean value $\frac{1}{x}\sum_{d\leq x} T(d)$ and the mean second moment
    $\overline{M}_{T,2}(x):=\frac{1}{x}\sum_{d\leq x} T(d)^2$ are $o(\sqrt{x})$ and $o(x)$, respectively,
    so that probably the bounds in Corollary~\ref{cor:1A} are not the sharpest ones. More precisely,
    their graphs suggest that they behave as $\sqrt{x}/(\log x)^{0.6}$, and as $x/(\log x)^{0.8}$
    respectively: see Figure~\ref{fig1}.

    \begin{fixedfig}
\begingroup
  \inputencoding{cp1252}%
  \makeatletter
  \providecommand\color[2][]{%
    \GenericError{(gnuplot) \space\space\space\@spaces}{%
      Package color not loaded in conjunction with
      terminal option `colourtext'%
    }{See the gnuplot documentation for explanation.%
    }{Either use 'blacktext' in gnuplot or load the package
      color.sty in LaTeX.}%
    \renewcommand\color[2][]{}%
  }%
  \providecommand\includegraphics[2][]{%
    \GenericError{(gnuplot) \space\space\space\@spaces}{%
      Package graphicx or graphics not loaded%
    }{See the gnuplot documentation for explanation.%
    }{The gnuplot epslatex terminal needs graphicx.sty or graphics.sty.}%
    \renewcommand\includegraphics[2][]{}%
  }%
  \providecommand\rotatebox[2]{#2}%
  \@ifundefined{ifGPcolor}{%
    \newif\ifGPcolor
    \GPcolorfalse
  }{}%
  \@ifundefined{ifGPblacktext}{%
    \newif\ifGPblacktext
    \GPblacktexttrue
  }{}%
  \let\gplgaddtomacro\g@addto@macro
  \gdef\gplbacktext{}%
  \gdef\gplfronttext{}%
  \makeatother
  \ifGPblacktext
    \def\colorrgb#1{}%
    \def\colorgray#1{}%
  \else
    \ifGPcolor
      \def\colorrgb#1{\color[rgb]{#1}}%
      \def\colorgray#1{\color[gray]{#1}}%
      \expandafter\def\csname LTw\endcsname{\color{white}}%
      \expandafter\def\csname LTb\endcsname{\color{black}}%
      \expandafter\def\csname LTa\endcsname{\color{black}}%
      \expandafter\def\csname LT0\endcsname{\color[rgb]{1,0,0}}%
      \expandafter\def\csname LT1\endcsname{\color[rgb]{0,1,0}}%
      \expandafter\def\csname LT2\endcsname{\color[rgb]{0,0,1}}%
      \expandafter\def\csname LT3\endcsname{\color[rgb]{1,0,1}}%
      \expandafter\def\csname LT4\endcsname{\color[rgb]{0,1,1}}%
      \expandafter\def\csname LT5\endcsname{\color[rgb]{1,1,0}}%
      \expandafter\def\csname LT6\endcsname{\color[rgb]{0,0,0}}%
      \expandafter\def\csname LT7\endcsname{\color[rgb]{1,0.3,0}}%
      \expandafter\def\csname LT8\endcsname{\color[rgb]{0.5,0.5,0.5}}%
    \else
      \def\colorrgb#1{\color{black}}%
      \def\colorgray#1{\color[gray]{#1}}%
      \expandafter\def\csname LTw\endcsname{\color{white}}%
      \expandafter\def\csname LTb\endcsname{\color{black}}%
      \expandafter\def\csname LTa\endcsname{\color{black}}%
      \expandafter\def\csname LT0\endcsname{\color{black}}%
      \expandafter\def\csname LT1\endcsname{\color{black}}%
      \expandafter\def\csname LT2\endcsname{\color{black}}%
      \expandafter\def\csname LT3\endcsname{\color{black}}%
      \expandafter\def\csname LT4\endcsname{\color{black}}%
      \expandafter\def\csname LT5\endcsname{\color{black}}%
      \expandafter\def\csname LT6\endcsname{\color{black}}%
      \expandafter\def\csname LT7\endcsname{\color{black}}%
      \expandafter\def\csname LT8\endcsname{\color{black}}%
    \fi
  \fi
    \setlength{\unitlength}{0.0500bp}%
    \ifx\gptboxheight\undefined%
      \newlength{\gptboxheight}%
      \newlength{\gptboxwidth}%
      \newsavebox{\gptboxtext}%
    \fi%
    \setlength{\fboxrule}{0.5pt}%
    \setlength{\fboxsep}{1pt}%
\begin{picture}(7936.00,3400.00)%
    \gplgaddtomacro\gplbacktext{%
      \csname LTb\endcsname
      \put(726,440){\makebox(0,0)[r]{\strut{}$0.94$}}%
      \put(726,831){\makebox(0,0)[r]{\strut{}$0.96$}}%
      \put(726,1223){\makebox(0,0)[r]{\strut{}$0.98$}}%
      \put(726,1614){\makebox(0,0)[r]{\strut{}$1$}}%
      \put(726,2005){\makebox(0,0)[r]{\strut{}$1.02$}}%
      \put(726,2396){\makebox(0,0)[r]{\strut{}$1.04$}}%
      \put(726,2788){\makebox(0,0)[r]{\strut{}$1.06$}}%
      \put(726,3179){\makebox(0,0)[r]{\strut{}$1.08$}}%
      \put(858,220){\makebox(0,0){\strut{}$0$}}%
      \put(2194,220){\makebox(0,0){\strut{}$200000$}}%
      \put(3530,220){\makebox(0,0){\strut{}$400000$}}%
      \put(4867,220){\makebox(0,0){\strut{}$600000$}}%
      \put(6203,220){\makebox(0,0){\strut{}$800000$}}%
      \put(7539,220){\makebox(0,0){\strut{}$1\times10^{6}$}}%
    }%
    \gplgaddtomacro\gplfronttext{%
    }%
    \gplbacktext
    \put(0,0){\includegraphics{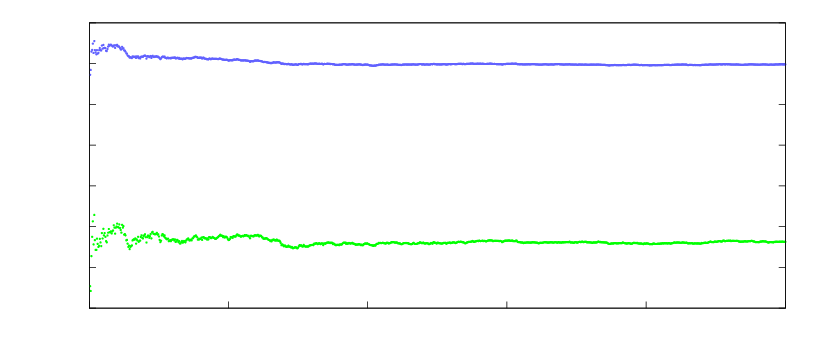}}%
    \gplfronttext
  \end{picture}%
\endgroup
\hspace{1cm}\mbox{}\\
    \caption{Graphs of $\big(\frac{1}{x}\sum_{d\leq x}T(d)\big)/\big(\sqrt{x}/(\log x)^{0.6}\big)$
    (upper) and of $\big(\frac{1}{x}\sum_{d\leq x}T(d)^2\big)/\big(x/(\log x)^{0.8}\big)$ (lower).}%
    \label{fig1}
    \end{fixedfig}
    \medskip

\item %
    The numerical evidence also shows that when restricted to prime values of $d$, the mean value
    $\frac{1}{\pi(x)}\sum_{d\leq x, d\,\text{prime}} T(d)$ and the mean second moment
    $\frac{1}{\pi(x)}\sum_{d\leq x, d\,\text{prime}} T(d)^2$ have the orders $\sqrt{x}$ and $x$,
    respectively: see Figure~\ref{fig2}. If correct, the comparison with what happens for the
    unrestricted means shows that, essentially, the largest values for $T(d)$ come from prime numbers.

    \begin{fixedfig}
\begingroup
  \inputencoding{cp1252}%
  \makeatletter
  \providecommand\color[2][]{%
    \GenericError{(gnuplot) \space\space\space\@spaces}{%
      Package color not loaded in conjunction with
      terminal option `colourtext'%
    }{See the gnuplot documentation for explanation.%
    }{Either use 'blacktext' in gnuplot or load the package
      color.sty in LaTeX.}%
    \renewcommand\color[2][]{}%
  }%
  \providecommand\includegraphics[2][]{%
    \GenericError{(gnuplot) \space\space\space\@spaces}{%
      Package graphicx or graphics not loaded%
    }{See the gnuplot documentation for explanation.%
    }{The gnuplot epslatex terminal needs graphicx.sty or graphics.sty.}%
    \renewcommand\includegraphics[2][]{}%
  }%
  \providecommand\rotatebox[2]{#2}%
  \@ifundefined{ifGPcolor}{%
    \newif\ifGPcolor
    \GPcolorfalse
  }{}%
  \@ifundefined{ifGPblacktext}{%
    \newif\ifGPblacktext
    \GPblacktexttrue
  }{}%
  \let\gplgaddtomacro\g@addto@macro
  \gdef\gplbacktext{}%
  \gdef\gplfronttext{}%
  \makeatother
  \ifGPblacktext
    \def\colorrgb#1{}%
    \def\colorgray#1{}%
  \else
    \ifGPcolor
      \def\colorrgb#1{\color[rgb]{#1}}%
      \def\colorgray#1{\color[gray]{#1}}%
      \expandafter\def\csname LTw\endcsname{\color{white}}%
      \expandafter\def\csname LTb\endcsname{\color{black}}%
      \expandafter\def\csname LTa\endcsname{\color{black}}%
      \expandafter\def\csname LT0\endcsname{\color[rgb]{1,0,0}}%
      \expandafter\def\csname LT1\endcsname{\color[rgb]{0,1,0}}%
      \expandafter\def\csname LT2\endcsname{\color[rgb]{0,0,1}}%
      \expandafter\def\csname LT3\endcsname{\color[rgb]{1,0,1}}%
      \expandafter\def\csname LT4\endcsname{\color[rgb]{0,1,1}}%
      \expandafter\def\csname LT5\endcsname{\color[rgb]{1,1,0}}%
      \expandafter\def\csname LT6\endcsname{\color[rgb]{0,0,0}}%
      \expandafter\def\csname LT7\endcsname{\color[rgb]{1,0.3,0}}%
      \expandafter\def\csname LT8\endcsname{\color[rgb]{0.5,0.5,0.5}}%
    \else
      \def\colorrgb#1{\color{black}}%
      \def\colorgray#1{\color[gray]{#1}}%
      \expandafter\def\csname LTw\endcsname{\color{white}}%
      \expandafter\def\csname LTb\endcsname{\color{black}}%
      \expandafter\def\csname LTa\endcsname{\color{black}}%
      \expandafter\def\csname LT0\endcsname{\color{black}}%
      \expandafter\def\csname LT1\endcsname{\color{black}}%
      \expandafter\def\csname LT2\endcsname{\color{black}}%
      \expandafter\def\csname LT3\endcsname{\color{black}}%
      \expandafter\def\csname LT4\endcsname{\color{black}}%
      \expandafter\def\csname LT5\endcsname{\color{black}}%
      \expandafter\def\csname LT6\endcsname{\color{black}}%
      \expandafter\def\csname LT7\endcsname{\color{black}}%
      \expandafter\def\csname LT8\endcsname{\color{black}}%
    \fi
  \fi
    \setlength{\unitlength}{0.0500bp}%
    \ifx\gptboxheight\undefined%
      \newlength{\gptboxheight}%
      \newlength{\gptboxwidth}%
      \newsavebox{\gptboxtext}%
    \fi%
    \setlength{\fboxrule}{0.5pt}%
    \setlength{\fboxsep}{1pt}%
\begin{picture}(7936.00,3400.00)%
    \gplgaddtomacro\gplbacktext{%
      \csname LTb\endcsname
      \put(726,440){\makebox(0,0)[r]{\strut{}$0.44$}}%
      \put(726,897){\makebox(0,0)[r]{\strut{}$0.46$}}%
      \put(726,1353){\makebox(0,0)[r]{\strut{}$0.48$}}%
      \put(726,1809){\makebox(0,0)[r]{\strut{}$0.5$}}%
      \put(726,2266){\makebox(0,0)[r]{\strut{}$0.52$}}%
      \put(726,2722){\makebox(0,0)[r]{\strut{}$0.54$}}%
      \put(726,3179){\makebox(0,0)[r]{\strut{}$0.56$}}%
      \put(858,220){\makebox(0,0){\strut{}$0$}}%
      \put(2194,220){\makebox(0,0){\strut{}$200000$}}%
      \put(3530,220){\makebox(0,0){\strut{}$400000$}}%
      \put(4867,220){\makebox(0,0){\strut{}$600000$}}%
      \put(6203,220){\makebox(0,0){\strut{}$800000$}}%
      \put(7539,220){\makebox(0,0){\strut{}$1\times10^{6}$}}%
    }%
    \gplgaddtomacro\gplfronttext{%
    }%
    \gplbacktext
    \put(0,0){\includegraphics{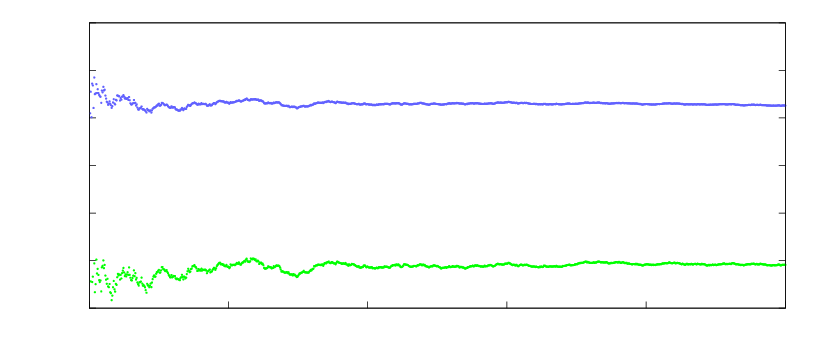}}%
    \gplfronttext
  \end{picture}%
\endgroup
\hspace{1cm}\mbox{}\\
    \caption{Graphs of $\big(\frac{1}{\pi(x)}\sum_{d\leq x,\, d\,\text{prime}}T(d)\big)/\sqrt{x}$ (upper)
    and of $\big(\frac{1}{\pi(x)}\sum_{d\leq x,\, d\,\text{prime}}T(d)^2\big)/x$ (lower).}%
    \label{fig2}
    \bigskip
    \end{fixedfig}
\item %
    Finally, we mention here that very recently M.~A.~Korolev has announced via a preprint in
    arXiv~\cite{Korolev} that $\sum_{d\leq x} g(d)^2 = (c_0 + O_\eps(x^{-1/108+\eps})) x^2$ with $c_0 =
    1.2\ldots$. As a consequence, the asymptotic behaviour for the second moment of $g$ is now known.
\end{List}

\begin{acknowledgments}
The authors wish to warmly thank Maciej Radziejewski for his assistance with the bibliographic research
through the many different editions of Sierpinski's book, and the anonymous referee for her/his careful
reading and suggestions. The first author is member of the INdAM research group GNAMPA, the second and
third authors of the INdAM group GNSAGA.
\end{acknowledgments}

\section{Proof of Equation~\texorpdfstring{\eqref{eq:2A}}{2}}
\label{sec:2}
\subsection{Preliminary results}
\label{subsec:2.1}
By its definition
\[
g(d) = \#\big\{(m,q)\in\N^2\colon m<\sqrt{d},\ |q-\sqrt{d}|<m, \text{ and } \exists k\in\Z \text{ with } d=m^2+kq\big\}.
\]
The strict inequalities in the definition of $g(d)$ require $m$, $q$, $k$ and $d$ to be at least $1$. Let
thus $\N_1 := \{n\in\N\colon n\geq 1\}$. Then
\[
g(d)=\#G(d),
\]
where
\begin{align*}
G(d)
&:= \big\{(m,q,k)\in\N_1^3\colon d=m^2+kq,|q-m|<\sqrt{d}<q+m\big\}
\intertext{that squaring becomes}
&\phantom:= \big\{(m,q,k)\in\N_1^3\colon d=m^2+kq,q-2m<k<q+2m\big\}.
\end{align*}
The sets $G(d)$ are disjoint because of the condition $d=m^2+kq$. Hence
\begin{equation}\label{eq:4A}
\sum_{d\leq x}g(d)=\#\bigcup_{d\leq x}G(d)
= \#\big\{(m,q,k)\in\N_1^3\colon m^2+kq\leq x,q-2m<k<q+2m\big\}.
\end{equation}
For future reference, we notice that if $(m,q,k)$ belongs to this set then $m\leq \sqrt{x-kq}<\sqrt{x}$
and, with $d:=m^2+kq$, from the initial definition of $G(d)$ we have $q<\sqrt{d}+m\leq\sqrt{x}+m$.\\
We need also the following lemma.
\begin{lemma}\label{lem:1A}
Let $y\geq 1$ be an integer. Then there exists $\theta = \theta(y) \in[0,1]$ such that:
\[
\sum_{q=y+1}^{2y}\frac{1}{q} = \log 2 - \frac{1}{4y + 2\theta}.
\]
\end{lemma}
\begin{proof}
Applying the third order Euler--McLaurin formula we get
\begin{align*}
\sum_{q=y+1}^{2y}\frac{1}{q}
&= \sum_{q=y}^{2y}\frac{1}{q} - \frac{1}{y}                  \\
&= \int_{y}^{2y}\frac{\dd q}{q} - \frac{1}{y}
  + \frac{1}{2}\Bigl[\frac{1}{2y} + \frac{1}{y}\Bigr]
  - \frac{1}{12}\Bigl[\frac{1}{(2y)^2} - \frac{1}{y^2}\Bigr]
  - \int_y^{2y} \frac{B_3(x)}{x^4}\dd x                      \\
&= \log 2
  - \frac{1}{4y}
  + \frac{1}{16y^2}
  - \int_y^{2y} \frac{B_3(x)}{x^4}\dd x,
\end{align*}
where, denoting $\fracpart{x}$ the fractional part of $x$, $B_3(x) :=
\fracpart{x}^3-3\fracpart{x}^2/2+\fracpart{x}/2$. The claim follows since $B_3(x)\in[-1/20,1/20]$.
\end{proof}

\subsection{Proof of~\texorpdfstring{\eqref{eq:2A}}{2}}
In all the following computations, each summation index (i.e. $m$, $q$ or $k$) is implicitly at least
equal to $1$.
Furthermore, to ease the notations, we set $y:=\intpartup{\sqrt{x}}-1$, hence $y<\sqrt{x}\leq y+1$, and
the symbol $\theta$ denotes a quantity which is not a constant, i.e. does not necessarily assume the same
value in every instance, but is always in $[0,1]$.

From~\eqref{eq:4A} and the lines below we have
\begin{align*}
\sum_{d\leq x} g(d)&
= \sum_{m<\sqrt{x}}
\Biggl[
    \sum_{q<\sqrt{x}-m} \sum_{k<q+2m}1
    + \sum_{\sqrt{x}-m\leq q<\sqrt{x}+m} \sum_{k\leq (x-m^2)/q}1
    - \sum_{2m<q<\sqrt{x}+m} \sum_{k\leq q-2m}1
\Biggr],                                                                       \\
\intertext{i.e.,}
=& \sum_{m<\sqrt{x}}
\Bigl[
    \sum_{q<\sqrt{x}-m} (q+2m-1)
    + \sum_{\sqrt{x}-m\leq q<\sqrt{x}+m} \Bigl(\frac{x-m^2}{q} -\theta\Bigr)
    - \sum_{ 2m<q<\sqrt{x}+m} (q-2m)
\Bigr],
\intertext{where $\theta$ depends on $x$, $m$ and $q$ but is in $[0,1]$.
The sums $\sum_{q<\sqrt{x}-m} q$ and $\sum_{2m<q<\sqrt{x}+m} (q-2m)$ cancel out. This gives}
=& \sum_{m<\sqrt{x}}
\Bigl[
    \sum_{q<\sqrt{x}-m} (2m-1)
    + \sum_{\sqrt{x}-m\leq q<\sqrt{x}+m} \frac{x-m^2}{q} -2\theta m
\Bigr]                                                                                     \\
=& \sum_{m<\sqrt{x}}
\Bigl[
    (2m-1)(y-m)
    + \sum_{\sqrt{x}-m\leq q<\sqrt{x}+m} \frac{x-m^2}{q}
\Bigr]
- \theta y(y+1).
\end{align*}
We exchange the summation order in the second sum. For this we observe that $m<\sqrt{x}$ means
that $m\leq y$, hence $q<\sqrt{x}+m\leq \sqrt{x}+y$ implies that $q\leq 2y$. Moreover, the assumption
$\sqrt{x}-m\leq q<\sqrt{x}+m$ gives $\sqrt{x}-q \leq m$ when $\sqrt{x}-q$ is positive, and $q-\sqrt{x} <
m$ otherwise. This gives
\begin{align*}
=&
     \sum_{m<\sqrt{x}} { (2m-1)(y-m) }
    + \Big[\sum_{q<\sqrt{x}} \sum_{\sqrt{x}-q\leq m<\sqrt{x}} + \sum_{\sqrt{x}\leq q\leq 2y} \sum_{q-\sqrt{x}<m<\sqrt{x}}\Big] \frac{x-m^2}{q}
    - \theta y(y+1).                                                                        \\
\intertext{In terms of $y$, the ranges become}
%
=&
     \sum_{m\leq y} { (2m-1)(y-m) }
    + \Big[\sum_{q\leq y} \sum_{y-q< m\leq y}
    + \sum_{y< q\leq 2y} \sum_{q-y\leq m\leq y}\Big] \frac{x-m^2}{q}
    - \theta y(y+1).
\intertext{Recalling that $\sum_{k=1}^N k^2 = N(N+1)(2N+1)/6$, after some algebra we get}
=& -\frac{2}{3}y^2+\frac{y}{6}
   + (2y+1)\Bigl(x-\frac{y^2+y}{3}\Bigr)
        \sum_{y< q\leq 2y}\frac{1}{q}
   - \theta y(y+1).
\intertext{By Lemma~\ref{lem:1A} this is}
=& -\frac{2}{3}y^2+\frac{y}{6}
   + (2y+1)\Bigl(x-\frac{y^2+y}{3}\Bigr)
        \Big(\log 2 - \frac{1}{4y+2\widetilde{\theta}}\Big)
   - \theta y(y+1),
\end{align*}
for some $\theta,\widetilde{\theta}\in [0,1]$. This is $\frac{4}{3}(\log 2)x^{3/2} + R$ with $R=O(x)$. With
elementary computations one shows that $-2x-2\sqrt{x} \leq R\leq -x+2\sqrt{x}$.

\section{Proof of Equation~\texorpdfstring{\eqref{eq:3A}}{3}}
\label{sec:3} %
From the computations at the beginning of~\ref{subsec:2.1} we have that $g(d)^2=\#(G(d)\times G(d))$ and
$G(d)$ are disjoint, so that
\begin{align*}
\sum_{d\leq x}g(d)^2
 =&\#\big\{(m_1,q_1,k_1,m_2,q_2,k_2)\in\N_1^6\colon
           m_1^2+k_1q_1=m_2^2+k_2q_2,\,m_1^2+k_1q_1\leq x\\
&\quad \text{ and }q_i-2m_i<k_i<q_i+2m_i\ \forall i\big\}.
\end{align*}
Thus,
\begin{align*}
\sum_{d\leq x} g(d)^2
&=
  \sum_{\substack{m_1,m_2\colon\\m_i<\sqrt{x}}}
  \sum_{\substack{q_1,q_2\colon\\ q_i<\sqrt{x}+m_i}}
  \sum_{\substack{k_1,k_2\colon\\
                  |q_i-k_i|<2m_i \\
                  m_1^2+k_1q_1=m_2^2+k_2q_2 \\
                  m_1^2+k_1q_1\leq x
                 }} 1
\intertext{that we estimate with}
&\leq
  \sum_{\substack{m_1,m_2\colon\\m_i<\sqrt{x}}}
  \sum_{\substack{q_1,q_2\colon\\q_i\leq \sqrt{x}+m_i}}
  \sum_{\substack{k_1,k_2\colon\\
                  0< k_i\leq 2\sqrt{x}-q_i  \\
                  m_1^2+k_1q_1=m_2^2+k_2q_2 \\
                 }} 1
\end{align*}
because $|q_i-k_i|<2m_i$ and $m_i^2+k_iq_i\leq x$ imply that $k_i\leq (x-m_i^2)/q_i \leq 2\sqrt{x}-q_i$:
the resulting bound is a bit loose, but it is independent of $m_i$, and this is very useful for the
computation.\\
We notice that
\[
  \sum_{\substack{k_1,k_2\colon\\
                  0 < k_1 \leq \beta_1\\
                  0 < k_2 \leq \beta_2\\
                  m_1^2+k_1q_1=m_2^2+k_2q_2 \\
                 }} 1
  \leq \Big(1+\min\Big(\frac{D\beta_1}{q_2},\frac{D\beta_2}{q_1}\Big)\Big)\chi_D(m_1^2,m_2^2)
\]
where $D:=\gcd(q_1,q_2)$ and $\chi_D(m_1^2,m_2^2) = 1$ when $m_1^2=m_2^2\pmod{D}$, $0$ otherwise.
\begin{proof}
Let $D:=\gcd(q_1,q_2)$. From the Chinese remainder theorem, there exist $k_1$ and $k_2\in\Z$ such that
$m_1^2+k_1q_1=m_2^2+k_2q_2$ if and only if $D\mid m_1^2-m_2^2$. In that case they are of the form
$k_1=a+\ell\frac{q_2}{D}$, $k_2=b+\ell\frac{q_1}{D}$, where $(a,b)$ is one solution and $\ell\in\Z$. The
assumption that $k_1$ is in $(0,\beta_1]$ implies that $\ell$ is in $(-aD/q_2,(\beta_1-a)D/q_2]$ and
there are at most $1+\beta_1 D/q_2$ integers in this interval.\\
The similar argument for $k_2$ gives the other upper bound for the number of solutions.
\end{proof}
As a consequence we have that
\begin{align*}
\sum_{d\leq x} g(d)^2
\leq&
  \sum_{\substack{m_1,m_2\colon\\m_i<\sqrt{x}}}
  \sum_{\substack{q_1,q_2\colon\\q_i\leq \sqrt{x}+m_i}}
  D\min\Big(\frac{2\sqrt{x}-q_1}{q_2},\frac{2\sqrt{x}-q_2}{q_1}\Big)\chi_D(m_1^2,m_2^2) \\
  &
  +\sum_{\substack{m_1,m_2\colon\\m_i<\sqrt{x}}}
   \sum_{\substack{q_1,q_2\colon\\q_i\leq \sqrt{x}+m_i}}\chi_D(m_1^2,m_2^2)
   =: U + V.
\end{align*}
Firstly we estimate the contribution of the $U$ term. We have
\begin{align}
U
&=
\sum_{\substack{q_1,q_2\colon\\q_i< 2\sqrt{x}}}
  D\min\Big(\frac{2\sqrt{x}-q_1}{q_2},\frac{2\sqrt{x}-q_2}{q_1}\Big)
  \sum_{\substack{m_1,m_2\colon\\\max(1,q_i-\sqrt{x})\leq m_i < \sqrt{x}}}
  \chi_D(m_1^2,m_2^2).                 \notag
\intertext{We split the computation according to the value of $D$. So, letting $q'_1:=q_1/D$,
$q'_2:=q_2/D$, and $L:= 2\sqrt{x}/D$ we get}
&=
\sum_{D\leq 2\sqrt{x}} D \sum_{\substack{q'_1,q'_2\colon\\q'_i< L\\ \gcd(q'_1,q'_2)=1}}
   \min\Big(\frac{L-q'_1}{q'_2},\frac{L-q'_2}{q'_1}\Big)
   \sum_{\substack{m_1,m_2\colon\\\max(1,Dq'_i-\sqrt{x})\leq m_i < \sqrt{x}}}
   \chi_D(m_1^2,m_2^2).                \label{eq:5A}
\end{align}
The next lemma allows to estimate the innermost sum.
\begin{lemma}\label{lem:2A}
Let $D\in\N$, $D\geq 1$, then
\[
\sum_{m_1,m_2\in\Z/D\Z}\chi_D(m_1^2,m_2^2) = c(D)D
\]
where $c$ is the multiplicative map with
\[
c(p^k):=
\begin{cases}
k          & \text{if }p=2,    \\
1+k-k/p    & \text{otherwise.}
\end{cases}
\]
\end{lemma}
\begin{proof}
By the Chinese remainder theorem $\chi_D(m_1^2,m_2^2)=1$ if and only if $m_1^2-m_2^2=0\pmod{p^k}$ for
every $p^k\| D$ so that it is sufficient to prove the result when $D=p^k$ is a power of a prime $p$.\\
If $n\in\Z$, we denote $\overline{n}$ the class of $n$ in $\Z/p^k\Z$. Let $A:=(\Z/p^k\Z)^2$ and
$P:=\bigl\{(\overline{m_1},\overline{m_2})\in A\colon \overline{m_1}^2=\overline{m_2}^2\bigr\}$. To prove
the lemma, we need to check that $P$ has $c(p^k)p^k$ elements. Let $f\colon A\longrightarrow A$ be the
map defined by
\[
f(\alpha,\beta):=(\alpha+\beta,\alpha-\beta),
\]
and let $m\colon A\longrightarrow\Z/p^k\Z$ be the multiplication map:
$m(\overline{a},\overline{b})=\overline{a}\overline{b}$. We thus have
\[
(\overline{m_1},\overline{m_2})\in P\iff m(f(\overline{m_1},\overline{m_2}))=\overline{0}.
\]
Let $Z:=m^{-1}(\overline{0})$. Let $(\overline{a},\overline{b})\in Z$. There are $p^k$ such pairs with
$\overline{a}=\overline{0}$. Otherwise, for any $u\in[0,k-1]$, there are $p^{k-u}-p^{k-u-1}$ classes
$\overline{a}$ such that $p^u\|a$; then $p^{k-u}\mid b$, and there are $p^u$ such classes; there are
therefore $p^k-p^{k-1}$ pairs $(\overline{a},\overline{b})\in Z$ with $p^u\|a$. Summing over $u$, and
adding the case with $\overline{a}=\overline{0}$, we get that $Z$ has $p^k+k(p^k-p^{k-1})$ elements.\\
If $p$ is odd, then $f$ is invertible with inverse
\[
g(\overline{a},\overline{b})
=\Bigl(\frac{\overline{a}+\overline{b}}{\overline{2}},\frac{\overline{a}-\overline{b}}{\overline{2}}\Bigr),
\]
therefore $P$ and $Z$ have the same number of elements.\\
If $p=2$, things are slightly more complicated because $2$ is not invertible anymore. To have
$\bigl(\overline{m_1},\overline{m_2}\bigr)\in P$, we need to have $m_1$ and $m_2$ of the same parity. Let
$E$ be the subset of $(\Z/2^k\Z)^2$ made of pairs $(\overline{m_1},\overline{m_2})$ where $m_1$ and $m_2$
have the same parity. Then $F:=f(E)$ is the set of pairs $(\overline{2a},\overline{2b})$ and
\[
f^{-1}\bigl(f(\overline{m_1},\overline{m_2})\bigr)
= \Big\{(\overline{m_1},\overline{m_2}),\big(\overline{m_1+2^{k-1}},\overline{m_2+2^{k-1}}\big)\Big\}.
\]
We have $P=f^{-1}(Z\cap F)$ and therefore the cardinality of $P$ is twice that of $Z\cap F$. If
$(\overline{a},\overline{b})\in Z\backslash F$, then $a$ or $b$ is odd. In that case $\overline{a}$ or
$\overline{b}$ is invertible so that the other one is $\overline{0}$. There are thus $2\cdot 2^{k-1}$
elements in $Z\backslash F$, therefore $k(2^k-2^{k-1})=k2^{k-1}$ elements in $Z\cap F$. Therefore $P$ has
$k2^k$ elements.
\end{proof}
According to the previous lemma we have that
\begin{align*}
   \sum_{\substack{m_1,m_2\colon\\\max(1,q_i-\sqrt{x})\leq m_i < \sqrt{x}}}
   \chi_D(m_1^2,m_2^2)
   &\leq c(D)D \prod_{i=1,2}\Big(\frac{\min(\sqrt{x}, 2\sqrt{x}-q_i)}{D} + 1\Big)\\
   &\leq c(D)D \prod_{i=1,2}\big(\min(L/2,L-q'_i) + 1\big),
\end{align*}
and from~\eqref{eq:5A}, with the suppression of the assumption $\gcd(q'_1,q'_2)=1$, we get
\begin{align}
U
%
%
&\leq
   \sum_{D\leq 2\sqrt{x}} c(D)D^2\sum_{\substack{q'_1,q'_2\colon\\q'_i < L}}
   \prod_{i=1,2}\big(\min(L/2,L-q'_i) + 1\big)
   \min\Big(\frac{L-q'_1}{q'_2},\frac{L-q'_2}{q'_1}\Big).               \label{eq:6A}
\end{align}
The function $f(q'_1,q'_2):=\prod_{i=1,2}\big(\min(L/2,L-q'_i) + 1\big) \min\big(\frac{L-q'_1}{q'_2},
\frac{L-q'_2}{q'_1}\big)$ decreases in each argument when $q'_1,q'_2 \in [1,L]$, so the corresponding
integral extended to the region $[0,L]\times[0,L]$ provides an upper bound for the sum. Moreover, we
further introduce the new variables $a,b$ such that $q'_1 =: aL = 2\sqrt{x}a/D$, $q'_2 =: bL =
2\sqrt{x}b/D$, yielding
\begin{align*}
U
&\leq
   \sum_{D\leq 2\sqrt{x}} c(D)D^2 \Big[\Big(\frac{2\sqrt{x}}{D}\Big)^4 A + \Big(\frac{2\sqrt{x}}{D}\Big)^3 B + \Big(\frac{2\sqrt{x}}{D}\Big)^2 C\Big],
\end{align*}
where
\begin{align*}
A &:=    \int_{a,b\in(0,1]}      \min\Big(\frac{1}{2},1-a\Big)\min\Big(\frac{1}{2},1-b\Big)
                                 \min\Big(\frac{1-a}{b},\frac{1-b}{a}\Big) \dd a \dd b = \frac{7}{6} \log 2 - \frac{37}{72},\\
B &:=    \int_{a,b\in(0,1]} \Big(\min\Big(\frac{1}{2},1-a\Big)+\min\Big(\frac{1}{2},1-b\Big)\Big)
                                 \min\Big(\frac{1-a}{b},\frac{1-b}{a}\Big) \dd a \dd b = \frac{19}{6}\log 2 - \frac{11}{12},\\
C &:=    \int_{a,b\in(0,1]}      \min\Big(\frac{1-a}{b},\frac{1-b}{a}\Big) \dd a \dd b = 2\log 2.
\end{align*}
%
Hence the bound says
\begin{align*}
U
&\leq
   16A x^2\sum_{D=1}^{+\infty} \frac{c(D)}{D^2} + 8B x^{3/2}\sum_{D\leq 2\sqrt{x}} \frac{c(D)}{D} + 4C x\sum_{D\leq 2\sqrt{x}} c(D),
\intertext{that we estimate with}
&\leq
   16A x^2\sum_{D=1}^{+\infty} \frac{c(D)}{D^2} + 8(B+C) x^{3/2}\sum_{D\leq 2\sqrt{x}} \frac{c(D)}{D},
\end{align*}
because $D\leq 2\sqrt{x}$.\\
We can evaluate explicitly the series, since $F(s):=\sum_{n=1}^{+\infty} c(n)n^{-s} = \frac{4^s-2^s+1}{4^s-2^{s-1}}
\frac{\zeta(s)^2}{\zeta(s+1)}$ for $\Ree(s)>1$. Thus,
\[
\sum_{D=1}^{+\infty} \frac{c(D)}{D^2}
 = F(2)
 = \frac{13}{14}\frac{\zeta(2)^2}{\zeta(3)}
 = 2.090\ldots
\]
To bound the sum we introduce both the Dirichlet series $H(s)=:\sum_{n=1}^{+\infty} h(n)n^{-s}$ such that
$F(s) =: H(s)\zeta(s)^2$, so that
\[
H(s)= \Big(1-\frac{1/2}{2^s} + \frac{3/2^2}{2^{2s}} + \frac{3/2^3}{2^{3s}} + \cdots \Big)\prod_{p\odd}\Big(1 - \frac{1/p}{p^s}\Big),
\]
and the series $\widetilde{H}(s):=\sum_{n=1}^{+\infty} |h(n)|n^{-s}$, so that
\[
\widetilde{H}(s)
= \Big(1+\frac{1/2}{2^s} + \frac{3/2^2}{2^{2s}} + \frac{3/2^3}{2^{3s}} + \cdots \Big)\prod_{p\odd}\Big(1 + \frac{1/p}{p^s}\Big)
= \frac{4^{s+1}+2}{4^{s+1}-1}\frac{\zeta(s+1)}{\zeta(2s+2)}.
\]
Thus,
\[
\sum_{D\leq 2\sqrt{x}} \frac{c(D)}{D}
= \sum_{u\leq 2\sqrt{x}} \frac{h(u)}{u} \sum_{v\leq 2\sqrt{x}/u}\frac{d(v)}{v},
\]
where $d(v)$ counts the divisors of $v$. Since $\sum_{v\leq w}\frac{d(v)}{v} \leq
\frac{1}{2}\log^2(e^2w)$ for every $w$ (by partial summation from $\sum_{v\leq w}d(v)\leq w\log w + w$,
which is weaker than what is known about the mean value for the divisor function, but which is sufficient
for our purposes),
%
we get
\begin{align*}
\sum_{D\leq 2\sqrt{x}} \frac{c(D)}{D}
&\leq \sum_{u\leq 2\sqrt{x}} \frac{|h(u)|}{2u} \log^2(2e^2\sqrt{x})
 \leq \frac{\widetilde{H}(1)}{2} \log^2(2e^2\sqrt{x})                    \\
&=    \frac{3}{5}\frac{\zeta(2)}{\zeta(4)} \log^2(2e^2\sqrt{x})
 \leq 0.228\log^2(4e^4x).
\end{align*}
Thus,
\begin{equation}\label{eq:7A}
U
\leq
   9.9\,x^2 + 4.87\,x^{3/2}\log^2(4e^4x).
\end{equation}
\noindent %
Lastly we estimate the contribution of the $V$ term. We bound it trivially, i.e. substituting
$\chi_D(m_1^2,m_2^2)$ with $1$. However, to improve the conclusion, we retain the remark that when $m_1$
and $m_2$ have different parity, then $q_1$ and $q_2$ cannot be both even. In this way we get that
\begin{align*}
V
&= \sum_{\substack{m_1,m_2\colon\\m_i<\sqrt{x}}}
   \sum_{\substack{q_1,q_2\colon\\q_i\leq \sqrt{x}+m_i}}\chi_D(m_1^2,m_2^2)
 \leq
   \sum_{\substack{m_1,m_2\colon\\m_i\leq\sqrt{x}}}
   \sum_{\substack{q_1,q_2\colon\\q_i\leq \sqrt{x}+m_i}}   1
   -
   \sum_{\substack{m_1,m_2\colon\\m_1\neq m_2 \pmod{2}\\ m_i\leq\sqrt{x}}}
   \sum_{\substack{q_1,q_2\colon\\q_i \even\\q_i\leq \sqrt{x}+m_i}}   1                   \\
&\leq \Big[\sum_{m\leq\sqrt{x}}(\sqrt{x}+m)\Big]^2
   -2\sum_{\substack{m_1 \leq \sqrt{x}\\m_1\vphantom{\odd}\even}}\sum_{\substack{m_2\leq \sqrt{x}\\m_2\odd}}
    \frac{(\sqrt{x}+m_1-2)}{2}\frac{(\sqrt{x}+m_2-2)}{2}                                  \\
&\leq \Big[\sum_{m\leq\sqrt{x}}(\sqrt{x}+m)\Big]^2
   -\frac{1}{2}
    \sum_{\substack{u \leq \frac{\sqrt{x}}{2}}}  (2u+\sqrt{x}-2)
    \sum_{\substack{v \leq \frac{\sqrt{x}+1}{2}}}(2v+\sqrt{x}-3).
\end{align*}
%
%
%
After some computations we get that
\begin{equation}\label{eq:8A}
V \leq \frac{63}{32}\,x^2 + \frac{51}{16}\,x^{3/2} - \frac{111}{32}\,x + \frac{57}{16}\,\sqrt{x} - \frac{5}{4}
  \leq \frac{63}{32}\,x^2 + \frac{51}{16}\,x^{3/2}.
\end{equation}
%
Adding~\eqref{eq:7A} and~\eqref{eq:8A} we get the claim since $\frac{51}{16}\,x^{3/2} \leq 0.13\,x^{3/2}
\log^2(4e^4x)$ for every $x\geq 1$.
%
\medskip

Retaining the coprimality condition in~\eqref{eq:6A} the main term in~\eqref{eq:7A} appears divided by
$\zeta(2)$, at the cost of error terms of size $O(x^{3/2}\log^4x)$. Also~\eqref{eq:8A} can be improved a
bit using the result in Lemma~\ref{lem:2A} and keeping the resulting coprimality condition. The combined
effect of these two improvements gives the bound $8 x^2 + O(x^{3/2}\log^4x)$ mentioned after
Theorem~\ref{th:1}.


\end{document}